\documentclass[11pt]{amsart}
\usepackage[margin=1.15in]{geometry}
\usepackage{amsmath,amssymb,amsthm,mathtools,mathrsfs,enumitem}
\usepackage{tikz-cd}
\usepackage[colorlinks=true,linkcolor=blue,citecolor=blue,urlcolor=blue]{hyperref}
\usepackage{cleveref}

\numberwithin{equation}{section}
\newtheorem{theorem}{Theorem}[section]

\newtheorem{lemma}[theorem]{Lemma}

\theoremstyle{definition}

\theoremstyle{remark}
\newtheorem{remark}[theorem]{Remark}

\newcommand{\Spec}{\operatorname{Spec}}

\newcommand{\im}{\operatorname{im}}

\newcommand{\bbA}{\mathbb A}

\newcommand{\cX}{\mathscr X}

\newcommand{\cU}{\mathscr U}
\newcommand{\cW}{\mathscr W}

\newcommand{\cZ}{\mathscr Z}

\newcommand{\SL}{\operatorname{SL}}

\newcommand{\id}{\mathrm{id}}
\newcommand{\schim}{\mathrm{Im_{sch}}}

\title[Open substacks with proper moduli spaces]{Finiteness of open substacks with proper moduli spaces}

\author{Xucheng Zhang}
\address{School of Mathematical Sciences, Beihang University, Beijing 102206, China}
\address{Email: xuchengz@buaa.edu.cn}

\date{}
\begin{document}

\begin{abstract}
For an algebraic stack of finite type, we construct a finite stratification of an atlas such that every universally closed open substack pulls back to a union of strata. As a result, there are only finitely many open substacks admitting a proper adequate moduli space.
\end{abstract}

\maketitle

\section{Introduction}
The choice of a stability condition is often the decisive step in constructing a proper moduli space. In general, a moduli stack $\cX$ may contain several open substacks with proper moduli spaces, for instance, by varying ample line bundles on $\cX$. In some cases, all such open substacks can be identified. For example, for the moduli stack of coherent sheaves on a curve, there is a unique open substack with a universally closed moduli space \cite{FHWZ25}. For a smooth quotient stack arising from a torus action, the resolution of the moment measure conjecture \cite{Heinloth-Zhang} classifies all opens with a proper good quotient; in particular, there are finitely many of them. The stack of all curves has many open substacks with proper coarse or good moduli spaces \cite{MR3044128}. Recently, for the moduli stack $\mathcal{G}_{1,n}$ parametrising Gorenstein, reduced curves of arithmetic genus one, all open substacks with proper good moduli spaces are classified \cite{BDLP26}.

In this short note, we show that such opens are always finite in number under very mild assumptions.
\begin{theorem}\label{thm:overview}
Let $S$ be a noetherian algebraic space and let $\cX \to S$ be an algebraic stack of finite type with qcqs diagonal. Then there are only finitely many open substacks $\cU \subseteq \cX$ for which $\cU \to S$ is universally closed. In particular, there are only finitely many open substacks of $\cX$ with a universally closed adequate moduli space.
\end{theorem}
\subsection{Idea of the proof}
The proof of \Cref{thm:overview} rests on a simple observation. Let $X$ be a connected $S$-scheme, and let $\cZ \subseteq \cX \times_S X$ be a closed substack, flat over $X$. For an open substack $\cU \subseteq \cX$ that is universally closed over $S$, the locus
\[
\{x \in X: \cZ_x \cap \cU_{\kappa(x)} \neq \varnothing\}
\]
is open by flatness and closed as $\cZ \cap (\cU \times_S X)$ is closed in $\cU \times_S X$. Thus this locus is either empty or all of $X$. \Cref{lem:closure-stratification} organises the point closures into such families over a finite stratification $X=\sqcup X_i$ of any fixed atlas $p: X \to \cX$. Since an open substack meets a point closure if and only if it contains the point, this observation shows that $p^{-1}(\cU)$ is a union of strata, giving the desired finiteness.
\begin{remark}
For reductive group acting on complex varieties, finiteness of the opens admitting proper good quotients follows from Bia{\l}ynicki-Birula's result \cite{BB98}; see \Cref{lem:finite-proper-good-quotients}. However, even for locally reductive stacks, an \'{e}tale reduction to quotient charts does not by itself suffice: the existence of a moduli space need not pass along an \'{e}tale morphism. Our approach here is different.
\end{remark}
\subsection{What's next}
Nat\"{u}rlich, the next thing is to determine which unions of strata define open substacks with proper moduli spaces, and it will give a full classification of such opens. This will be addressed in the forthcoming paper \cite{generalized-mm}. Our strategy is check which unions are $\Theta$-reductive, S-complete and satisfy the existence part of the valuative criterion for properness \cite{AHLH23}. Such a union is encoded by a choice function $m: \{X_i\} \to \{0,1\}$, and we will translate these criterion on choice functions. In particular, for a quotient stack $[X/G]$, the natural atlas $X \to [X/G]$ allows these conditions to be studied on open $G$-invariant unions of strata in $X$ itself, which can be seen as a generalization of the torus case \cite{Heinloth-Zhang}.

\section{The key lemma}
The following lemma is a relative, stack-theoretic analogue of the classical construction of flat families of orbit closures; compare \cite[\S1.1]{Thaddeus99} and \cite[Lem. 7.1 and the subsequent paragraph]{MR3676045}. We include a proof for the stated generality.
\begin{lemma}[Flat families of point closures]\label{lem:closure-stratification}
Let $S$ be a noetherian scheme and let $\cX \to S$ be an algebraic stack of finite type with qcqs diagonal. Let $p: X \to \cX$ be an atlas, where $X$ is a scheme of finite type over $S$. Then there exist a finite stratification $X=\sqcup_i X_i$ by integral locally closed subschemes, and closed substacks
\[
\cZ_i \subseteq \cX \times_S X_i
\]
such that $\cZ_i \to X_i$ is flat and of finite presentation,
and for every point $x: \Spec K \to X_i$
\[
(\cZ_i)_x=\schim(p_x) \subseteq \cX_K,
\]
where $p_x: \Spec K \to \cX_K$ is the morphism induced by $p \circ x$.
\end{lemma}
\begin{proof}
Let $R:=X \times_{\cX} X$ with projections $s,t: R \rightrightarrows X$. Consider
the Cartesian diagram
\[
\begin{tikzcd}
R \ar[rr,"{(s,t)}"] \ar[d,"t"'] & & X \times_S X \ar[d,"{p \times \id_X}"] \\
X \ar[rr,"{(p,\id_X)}"'] & & \cX \times_S X.
\end{tikzcd}
\]
The hypothesis on diagonal implies that $(p,\id_X)$ is quasi-compact and that $R$ is qcqs. Since $t$ is smooth, it is of finite presentation. Choose an \'{e}tale surjection\footnote{If $\cX$ has affine diagonal, then $R$ is a scheme of finite presentation over $X$, and we need not pass to $V$.} $V \twoheadrightarrow R$ with $V$ a scheme of finite presentation over $X$, and regard
\[
a: V \to R \xrightarrow{(s,t)} X \times_S X
\]
as a morphism over the second $X$ factor. By \cite[\href{https://stacks.math.columbia.edu/tag/0H41}{Tag 0H41}]{sp}, there is a finite stratification of $X$ over which formation of the scheme-theoretic image of $a$ commutes with arbitrary base change. Refining this stratification using generic flatness \cite[\href{https://stacks.math.columbia.edu/tag/052A}{Tag 052A}]{sp} and noetherian induction, we obtain finitely many integral locally closed strata $X_i$ such that
\[
V_i:=\schim \bigl( V \times_X X_i \to X \times_S X_i \bigr)
\]
is flat over $X_i$. Compatibility with arbitrary base change
is preserved under these refinements. Let $\cZ_i \subseteq \cX \times_S X_i$ be the scheme-theoretic image
of
\[
(p|_{X_i},\id_{X_i}): X_i \to \cX \times_S X_i.
\]
Scheme-theoretic images are unchanged by pre-composing with
a smooth surjection and, for quasi-compact morphisms, commute
with flat base change
\cite[\href{https://stacks.math.columbia.edu/tag/0CMJ}{Tag 0CMJ} and \href{https://stacks.math.columbia.edu/tag/0CMK}{Tag 0CMK}]{sp}. Hence
\[
\cZ_i \times_{\cX \times_S X_i}(X \times_S X_i)=V_i.
\]
Therefore, $\cZ_i \to X_i$ is flat by smooth descent.
It is of finite type and quasi-separated over the noetherian
scheme $X_i$, hence of finite presentation. The same comparison then shows that formation of $\cZ_i$ commutes with arbitrary base change over $X_i$. For any point $x \in X_i(K)$, the fibre $(\cZ_i)_x$ is the scheme-theoretic image of $p_x: \Spec K \to \cX_K$. In particular, by \cite[\href{https://stacks.math.columbia.edu/tag/0CML}{Tag 0CML}]{sp} it follows that $|(\cZ_i)_x|=\overline{\{p_x\}}$.
\end{proof}

\section{Proof of \Cref{thm:overview}}
\begin{proof}[Proof of \Cref{thm:overview}]
Choose an \'{e}tale surjection $S' \twoheadrightarrow S$ with $S'$ a noetherian scheme. Pullback preserves universal closedness and distinguishes open substacks, so we may and do assume that $S$ is a scheme.

Choose an atlas $p: X \to \cX$ with $X$ a scheme of finite type over $S$. Fix a finite stratification $X=\sqcup_i X_i$ and a family of closed substacks $\cZ_i \subseteq \cX \times_S X_i$ as in \Cref{lem:closure-stratification}. Let $\cU \subseteq \cX$ be an open substack such that $\cU \to S$ is universally closed. For each $i$, let $\cW_i:=\cZ_i \times_{\cX} \cU$ and consider the commutative diagram
\[
\begin{tikzcd}
\cW_i \ar[r] \ar[d,hook] \ar[ddd,bend right=4.5em,"q_i"'] & \cU \ar[dd,hook] \ar[ddd,bend left=4.5em,"\text{univ. closed}"] \\
\cZ_i \ar[d,hook] & \\
\cX \times_S X_i \ar[r] \ar[d] & \cX \ar[d] \\
X_i \ar[r] & S
\end{tikzcd}
\]
The projection $q_i: \cW_i \to X_i$ factors as the closed immersion $\cW_i \hookrightarrow \cU \times_S X_i$ followed by the base change of $\cU \to S$, so it is universally closed. It is also flat and locally of finite presentation, being the restriction of $\cZ_i\to X_i$ to an open substack. Hence $\im(q_i)$ is open and closed in the integral scheme $X_i$, and is therefore either empty or all of $X_i$.

To conclude, we claim that $\im(q_i)=p^{-1}(\cU) \cap X_i$ and hence $p^{-1}(\cU)$ is a union of strata. Indeed, for any point $x: \Spec K \to X_i$, let $p_x \in |\cX_K|$ be the point represented by the morphism $p \circ x$. By \Cref{lem:closure-stratification}, we have
\[
(\cW_i)_x \neq \varnothing \Leftrightarrow \overline{\{p_x\}} \cap \cU_K \neq \varnothing \Leftrightarrow p_x \in \cU_K,
\]
and the last equivalence follows because $\cU_K$ is open in $\cX_K$. The last condition is equivalent to $x \in p^{-1}(\cU)$, proving the claim.
\end{proof}
For quotient stacks $[X/G]$ over the field $\mathbb{C}$ of complex numbers, finiteness of opens with a proper moduli space can be alternatively deduced from \cite{BB98}.
\begin{lemma}
\label{lem:finite-proper-good-quotients}
Let $G$ be a reductive group over $\mathbb{C}$, and let $X$ be a complex variety with a $G$-action. Then there are only finitely many $G$-invariant open subsets $U \subseteq X$ that admits a proper good quotient $q_U: U \to Q_U$.
\end{lemma}

\begin{proof}
By \cite[Thm. I]{BB98}, there are only finitely many $G$-invariant open subsets admitting separated good quotients
that are maximal with respect to saturated inclusion. It suffices to show that every open subset admitting a proper good quotient is maximal in this sense. 

Let $U \subseteq V \subseteq X$ be $G$-invariant open subsets with good quotients
\[
q_U: U \to Q_U \text{ and } q_V: V \to Q_V,
\]
such that $Q_U$ is proper, $Q_V$ is separated, and $U$ is saturated with respect to $q_V$, meaning that
\[
U=q_V^{-1}\bigl(q_V(U)\bigr).
\]
Then the restriction of $q_V$ to $U$ identifies $Q_U$ with the open subspace $q_V(U) \subseteq Q_V$. Thus the induced morphism $Q_U \to Q_V$ is an open immersion. It is also proper, since $Q_U$ is proper and $Q_V$ is separated. Its image is therefore both open and closed. Since $Q_V$ is irreducible, this image is all of $Q_V$, and saturation gives $U=V$.
\end{proof}
\begin{remark}
However, the na\"{i}ve generalization from quotient stacks to general stacks seems not to work, even if the stack is locally reductive: the existence of an \'{e}tale morphism from a stack to a stack with an adequate moduli space (which may even be proper) does not guarantee that the former admits an adequate moduli space. Even a representable open immersion gives a counterexample:
\[
\mathrm{B}\mathbb{G}_a \simeq [(\bbA^2 \setminus \{0\})/\SL_2] \hookrightarrow [\bbA^2/\SL_2].
\]
\end{remark}

\bibliographystyle{amsalpha}
\bibliography{references}

\end{document}